\documentclass[11pt]{article}
\usepackage{amssymb,amsfonts,amsmath,latexsym,epsf,tikz,url}
\usepackage[usenames,dvipsnames]{pstricks}
\usepackage[lined,ruled,linesnumbered]{algorithm2e}
\usepackage{pstricks-add}
\usepackage{epsfig}
\usepackage{pst-grad} 
\usepackage{pst-plot} 
\usepackage[space]{grffile} 
\usepackage{etoolbox} 
\usepackage{float}
\usepackage{soul}
\usepackage{tikz}
\usepackage[colorinlistoftodos]{todonotes}
\usepackage{pgfplots}
\usepackage{mathrsfs}
\usepackage[colorlinks]{hyperref}
\usetikzlibrary{arrows}
\makeatletter 
\patchcmd\Gread@eps{\@inputcheck#1 }{\@inputcheck"#1"\relax}{}{}
\makeatother

\newtheorem{theorem}{Theorem}[section]

\newtheorem{observation}[theorem]{Observation}
\newtheorem{corollary}[theorem]{Corollary}
\newtheorem{lemma}[theorem]{Lemma}

\newtheorem{definition}[theorem]{Definition}

\newcommand{\qed}{\hfill $\square$\medskip}
\newcommand{\ve}{{\bf VE}}      
\newcommand{\spo}{{\bf support}}     
\begin{document}

\def\nt{\noindent}

\title{Connected coalitions in graphs}

\author{\small 	
	Saeid Alikhani$^{1}$,  
     Davood Bakhshesh$^2$, 
	 Hamidreza Golmohammadi$^{3}$,
    Elena V. Konstantinova$^{3,4}$
}


\maketitle

\begin{center}
	
	$^1$Department of Mathematical Sciences, Yazd University, 89195-741, Yazd, Iran\\

     $^{2}$Department of Computer Science, University of Bojnord, Bojnord, Iran\\ 
	
	$^{3}$Novosibirsk State University, Pirogova str. 2, Novosibirsk, 630090, Russia\\ 
	
	$^4$Sobolev Institute of Mathematics, Ak. Koptyug av. 4, Novosibirsk,
	630090, Russia\\
	\medskip
	{\tt alikhani@yazd.ac.ir ~~  d.bakhshesh@ub.ac.ir ~~h.golmohammadi@g.nsu.ru ~~e\_konsta@math.nsc.ru}
\end{center}

\begin{abstract}
    The connected coalition in a graph $G=(V,E)$ consists of two disjoint sets of vertices $V_{1}$ and $V_{2}$, neither of which is a connected dominating set but whose union $V_{1}\cup  V_{2}$, is a connected dominating set. A connected coalition partition in a
 graph $G$ of order $n=|V|$ is a vertex partition $\psi$ = $\{V_1, V_2,..., V_k \}$ such that every set $V_i \in \psi$ either
is a connected dominating set consisting of a single vertex of degree $n-1$, or is not a connected dominating set but
forms a connected coalition with another set $V_j\in \psi$ which is not a connected dominating set. 
 The connected coalition number, denoted by $CC(G)$, 
is the maximum cardinality of a connected coalition partition of $G$.
  In this paper, we initiate the study of    connected coalition in graphs, and  present some basic results. Precisely, we characterize all graphs that  have a  connected coalition partition. Moreover, we show that for any graph $G$ of order $n$ with $\delta(G)=1$ and with no full vertex, it holds that $CC(G)<n$. Furthermore, we show that for  any tree $T$, $CC(T)=2$. Finally, we present two polynomial-time algorithms that  for a given connected graph $G$ of order $n$  determine whether $CC(G)=n$ or $CC(G)=n-1$.
\end{abstract}


\medskip
\noindent{\bf AMS Subj.\ Class.:}  05C60.


\section{Introduction} 

Let $G = (V,E)$ be a simple graph with the vertex set $V$ and the edge set $E$. The {\em open neighborhood} of a vertex $v \in V$ is defined as $N(v) = {u \mid uv \in E}$, and its {\em closed neighborhood} as $N[v] = N(v) \cup {v}$. Each vertex $u \in N(v)$ is referred to as a {\em neighbor} of $v$, and $|N(v)|$ is referred to as the {\em degree} of $v$, denoted by $deg(v)$. A vertex $v$ in $G$ is considered {\em pendant} if its open neighborhood $N(v)$ has only one vertex. This vertex is referred to as the {\em support vertex} of $v$, denoted by $\spo(v)$. An edge is considered pendant if one of its vertices is pendant. In a tree $T$, a vertex of degree one is referred to as a {\em leaf}, and the vertex adjacent to it is referred to as a support vertex. The set of leaves in a tree $T$ is denoted by $L(T)$, and its cardinality by $l(T)$. In a graph $G$ with $n = |V|$ vertices, a vertex of degree $n-1$ is referred to as a {\em full} or {\em universal vertex}, and a vertex of degree $0$ is referred to as an {\em isolate}. The minimum degree of a graph $G$ is denoted by $\delta(G)$, and the maximum degree by $\Delta(G)$. A subset $V_{i} \subseteq V$ is referred to as a {\em singleton} set if $|V_i| = 1$, or a non-singleton set if $|V_i| \ge 2$. A graph without any isolated vertices is referred to as a non-isolated graph. 

A set $S \subseteq V$ in a graph $G$ is considered to be a {\em dominating} set if every vertex in $V \setminus S$ has at least one neighbor in the set $S$. A {\em connected dominating} set $D$ is defined as a dominating set such that the subgraph induced by the vertices in $D$ is connected. The {\em connected domination number}, denoted by $\gamma_{c}(G)$, is the minimum size of a connected dominating set in the graph $G$. Connected domination was first introduced in 1979 by Sampathkumar and Walikar, based on a suggestion by S.T. Hedetniemi \cite{13}. Over the past two decades, it has garnered significant interest due to its important applications in Wireless Sensor Networks. For more information, readers are referred to the literature including \cite{10,11,12}.

A {\em domatic partition} is a partition of the vertex set into dominating sets. The connected domatic partition is a similar partition into connected dominating sets. The maximum size of a domatic partition is called the {\em domatic number}, denoted by $d(G)$. The maximum size of a connected domatic partition is called the {\em connected domatic number}, denoted by $d_{c}(G)$. The domatic number was first introduced by Cockayne and Hedetniemi \cite{4}, and the connected domatic number was introduced by Zelinka in \cite{16}. Further information on these concepts can be found in sources such as \cite{6,14,15,16}.

The idea of coalitions and coalition partitions was first introduced in \cite{7} and has since then been studied in the field of graph theory, as seen in works such as \cite{2,3,8,9}. The definition of coalitions and coalition partitions were based on general graph properties, but the main focus was on their relation to the concept of dominating sets. A {\em coalition} $\pi$ in a graph $G$ is defined as two disjoint sets of vertices, $V_1$ and $V_2$, that individually cannot dominate the graph, but their union $V_1 \cup V_2$ is able to dominate the graph. The sets $V_1$ and $V_2$ are referred to as {\em coalition partners} in $\pi$. A {\em coalition partition}, also known as $c$-partition, in a graph $G$ is a partition of the vertices of $G$ into sets $\pi = \{V_1, V_2, \ldots, V_k\}$ such that each $V_i$ in $\pi$ is either a singleton dominating set of $G$ or a non-dominating set that forms a coalition with another non-dominating set $V_j \in \pi$. The {\em coalition number} $C(G)$ of a graph is the maximum number of sets that can be present in a $c$-partition of $G$. A $c$-partition of $G$ with $C(G)$ sets is referred to as a $C(G)$-partition.

For every $c$-partition $\pi$ of a graph $G$, there is a corresponding graph called the {\em coalition graph} of $G$ with respect to $\pi$, denoted as $CG(G, \pi)$. The vertices of this graph correspond one-to-one with the sets of $\pi$, and two vertices are adjacent in $CG(G, \pi)$ if and only if their corresponding sets form a coalition. The study of coalition graphs, particularly for paths, cycles, and trees, was conducted in \cite{8}. The concept of total coalition was introduced and explored in \cite{1}, while the coalition parameter for cubic graphs of order at most 10 was investigated in \cite{2}.

According to Section 4 of reference \cite{7}, there are open problems and areas for future research which suggest exploring connected dominating $c$-partition. Inspired by this, our focus is on the examination of connected coalitions and their partitions.

   
 In Section 2, we define and  discuss some properties of connected coalitions. In Section 3, we  determine the connected coalition number of graphs with pendant edge. Furthermore, we consider the connected coalition of trees in Section 4. In Section 5, we present two polynomial-time algorithms that  for a given connected graph $G$ of order $n$  determine whether $CC(G)=n$ or $CC(G)=n-1$. Finally, we present some open problems for future works in Section 6. 
 
\section{Introduction to connected coalition}

In this section, we first state the definition of the connected coalition and connected coalition partition.
 
 \begin{definition}[Connected coalition]
For a graph $G$ with vertex set $V$, two sets $V_1,V_2\subseteq V$ form a {connected coalition}, if neither $V_1$ nor $V_2$ is a connected dominating set but $V_1\cup V_2$ is a connected dominating set in $G$. 
 \end{definition} 
 
\begin{definition}[Connected coalition partition]\label{2.2} 
A connected coalition partition in a graph $G$ is a vertex partition $\psi$ = $\{V_1, V_2,\ldots, V_k \}$ such that every set $V_i$ of $\psi$ either
is a connected dominating set consisting of a single vertex of degree $n-1$, or is not a connected dominating set but
forms a connected coalition with another set $V_j\in \psi$ which is not a connected dominating set. The maximum cardinality of a connected coalition partition of $G$ is called the {\it connected coalition number} of $G$, denoted by $CC(G)$. A connected $c$-partition of a graph $G$ with the cardinality $CC(G)$ is denoted by $CC(G)$-partition. 
\end{definition}

Considering the graph $G$ should be connected, we have the following trivial observation.
\begin{observation}
\label{obs1}
For any disconnected graph $G$ of order $n\geq 2$, we have $CC(G)=0$.
\end{observation}
Now, by the following Theorem we characterize the graphs $G$ having $C(G)=1$.
\begin{lemma}
\label{lemk1}
For any graph $G$, $CC(G)=1$ if and only if $G=K_1$.
\end{lemma}
\begin{proof}
If  $CC(G)=1$, then $\{V\}$ is a $CC(G)$-partition. By Definition \ref{2.2},  we must have $|V|=1$. So, it is clear that $G=K_1$. Conversely, if $G=K_1$, clearly we have $CC(G)=1$.\qed
\end{proof}

Now, we prove the following lemma.
\begin{lemma} \label{gelem}
	 If $G$ is a connected graph of order $n>1$ with no full vertex, then $CC(G)\ge 2d_c(G)$.
\end{lemma} 

\begin{proof} 
	 Let $\mathcal{D}$ = $\{D_1, D_2,\ldots, D_k \}$ be a connected domatic partition of a
graph $G$, where $k=d_c(G)$. Since $G$ has no full vertex, all sets $D_i$ are not singletons. Now, we can suppose that
$D_1, D_2, \ldots, D_{k-1}$ are minimal connected dominating sets of $G$; if any set
$D_{i}$, for $1\leq i \leq k-1$, is not a minimal connected dominating set, let $D'_i\subseteq D_i$
be a minimal connected dominating set contained in $D_i$ and
add the vertices in $D_i\setminus D'_i$ to $D_k$. Note that any  partition of a non-singleton,
minimal connected dominating set into two nonempty sets creates two  non-connected dominating sets whose union
forms a connected coalition. Therefore, for every $1\leq i \leq k-1$, we
can partition every non-singleton set $D_i$ into two sets $D_{i,1}$, 
and $D_{i,2}$, which form a connected coalition. Doing this for each $D_i$, for $1\leq i \leq k-1$,
we obtain a collection $\mathcal{D'}$ of sets, each of
which is either a connected dominating set consisting of a single vertex of degree $n-1$ or is a non-connected dominating set that forms a coalition with another non-connected dominating set in $\mathcal{D'}$.
Now we can consider the connected dominating set $D_k$. If $D_k$ is a minimal
connected dominating set, then we can partition it into two non-connected dominating sets, add these two sets to $\mathcal{D'}$ and create a connected c-partition of G of order at least $k+1 > d_{c}(G)$. 
 If $D_k$ is not a minimal connected dominating set, let $D'_k\subseteq D_k$ be a minimal connected dominating set contained in $D_k$, partition $D'_k=D'_{k,1}\cup D'_{k,2}$ into two non-empty, non-connected dominating sets, which
together form a connected coalition, and let $D''_k= D_k \setminus D'_k$ and add $D'_{k,1}$ and $D'_{k,2}$ to $\mathcal{D'}$. 
It follows that $D''_k$ is not a connected dominating set, else there are
at least $k + 1$ disjoint connected dominating sets in $G$, a contradiction,
since $k=d_{c}(G)$.
 If $D''_k$ forms a connected coalition with any non-connected dominating set, then adding $D''_k$ to $\mathcal{D'}$, we have a connected c-partition of $G$ of order at least
$k+2 > d_{c}(G)$. However, if $D''_k$ does not
form a connected coalition with any set in $\mathcal{D'}$, then remove $D'_{k,2}$
from $\mathcal{D'}$ and add the set $D'_{k,2}\cup D''_k$, to $\mathcal{D'}$, 
therefore we create a connected c-partition of $G$ of order at least $k+1 > d_{c}(G)$. 
Now, based on  above construction of a connected $c$-partition of $G$, the number of the elements of this partition is minimized when each set $D_i$ ($1\leq i\leq k$) is partitioned into  two sets. Hence, it holds that $CC(G)\geq 2d_c(G)$.
\qed
 \end{proof}

 It is remarkable  that for any graph $G$, $d_c(G)\geq 1$. Based on Lemma \ref{gelem}, we have the following result.
\begin{theorem}\label{ge2}
If $G$ is a connected graph of order $n>1$ with no full vertex, then $CC(G)\ge 2$.
\end{theorem}

By Theorem \ref{ge2}, we immediately conclude the following result.
\begin{corollary}
\label{corlesstwo}
If  $G$ is a connected graph with $CC(G)<2$, then $G$ has  at least one full vertex.
\end{corollary}

 Our immediate aim in this paper is to investigate the possibility of the existence of a connected c-partition of a graph $G$. For this purpose, we define a family of a graphs, denoted by $\mathcal{F}$, in the following.\newline
 For any two graphs $G$ and $H$,  let $G+H$ be the {\em join} of two graphs G and H which is a graph constructed from disjoint copies of
$G$ and $H$ by connecting each vertex of $G$ to each vertex of $H$. Now, we state the following definition.
\begin{definition}
\label{family}
A family ${\cal F}$ of graphs is constructed as follows:
\begin{itemize}
\item {\bf Step 1.} We add all disconnected graphs $G$ of order $n\geq 2$ into ${\cal F}$. 
\item {\bf Step 2.} For any graph $G\in {\cal F}$, we add $G+K_1$ into ${\cal F}$.
\end{itemize}
\end{definition}

It is remarkable that the family $\cal F$ contains both many disconnected graphs and many connected graphs. For instance, Figure \ref{e1pic} shows a connected graph in $\cal F$. As another example,  consider the the friendship graphs $F_n$ which is a   graph with $2n + 1$ vertices and $3n $ edges, formed by the join of $K_1+nK_{2}$. Based on Definition \ref{family}, we have ${F}_n\in{\cal F}$.

\begin{figure}[ht]
\begin{center}
	\includegraphics[width=0.3\linewidth]{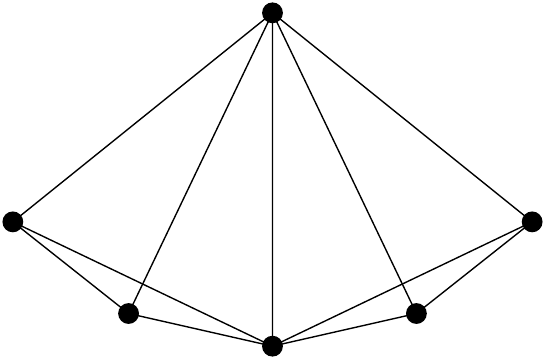}
	\caption{A connected graph in $\cal F$.
}
  \label{e1pic}
	\end{center}
\end{figure}

Now, we prove the following Lemma. 
\begin{lemma}
\label{calF}

For any graph $G$, if $G\in {\cal F}$, then $CC(G)=0$.
\end{lemma}
\begin{proof}
Using induction on the number of full vertices of $G$, we prove that $CC(G)=0$. For the base step, if $G$ has no full vertex, then since $G\in {\cal F}$, $G$ is a disconnected graph of order $n\geq 2$. By Observation \ref{obs1}, we have $CC(G)=0$. For the induction hypothesis step, suppose that for any graph $H\in {\cal F}$ such that the number of its full vertices is less than the number of full vertices  of $G$, it holds that $CC(H)=0$.  For induction step, suppose that $G=H+K_1$, where $H\in {\cal F}$.  Let $u$ be the vertex of $K_1$. By induction hypothesis, we have $CC(H)=0$. Now, if  $CC(G)\neq 0$,  since $G\neq K_1$, by Lemma \ref{lemk1}, we must have $CC(G)\geq2$. Since $u$ is the full vertex of $G$, the set $\{u\}$ belongs to any $CC(G)$-partition $\psi$. Now, by removing $\{u\}$ of $\psi$, we obtain a $c$-partition for $H$ with $CC(H)\geq 1$, which is a contradiction. Thus, $CC(G)=0$. \qed
\end{proof}

\medskip
The next theorem shows a necessary and sufficient condition for the existence of a connected $c$-partition of a graph $G$.
\begin{theorem}\label{1}
For any graph $G$, $CC(G)=0$ if and only if $G\in {\cal F}$. 
\end{theorem}

\begin{proof}
By Corollary \ref{corlesstwo}, since we assume that $C(G)<2$,   $G$ has at least one full vertex. Now, if $G\in {\cal F}$, by Lemma \ref{calF}, we have $CC(G)=0$, Conversely, suppose that $CC(G)=0$.  To prove $G\in {\cal F}$, we use the induction on the number of full vertices of $G$. For the base step, we assume that $G$ contains exactly one full vertex $u$. Now, consider the graph $G'=G[V\backslash \{u\}]$. If $G'$ is a connected graph, since $G'$ has no full vertex, by Theorem \ref{ge2}, $CC(G')\geq 2$. Hence, using a $CC(G')$-partition and the singleton set $\{u\}$, we can construct a $CC(G)$-partition with $CC(G)\geq 3$, which is a contradiction.  Hence, $G'$ must be disconnected.  Hence, by the definition of $\cal F$, we have $G\in {\cal F}$. 

For induction hypothesis, we assume that if  $G'$ is a connected graph  with~$CC(G')=0$ such that  the number of its full vertices is less than $G$, then $G'\in {\cal F}$. 

Now, we prove the induction step.  Let $u$ be the full vertex of $G$. Consider the graph $G'=G[V\backslash\{u\}]$. Now, we have two cases. 
\begin{itemize}
\item{\bf Case 1.} $G'$ is disconnected. Then, by the definition of $\cal F$, $G\in {\cal F}$.
\item{\bf Case 2.} $G'$ is connected. Since $CC(G)=0$, we must have $CC(G')=0$, because otherwise, similar to the above arguments, using a $CC(G')$ -partition and the singleton set $\{u\}$, we can construct a $CC(G)$-partition with $CC(G)\geq 3$, which is a contradiction. Now, since $G'$ is connected and $CC(G')=0$, by Corollary~\ref{corlesstwo},  $G'$ has at least one full vertex. It is clear that the number of full vertices of $G'$ is less than the number of full vertices of $G$. Then, by induction hypothesis, $G'\in {\cal F}$. Hence, by the definition of $\cal F$, we can see $G\in {\cal F}$. This completes the proof. \qed
\end{itemize}
\end{proof}


By Theorem \ref{1} and  Lemma \ref{lemk1}, we have the following corollaries.

\begin{corollary}\label{cor1}
If $G\not\in\mathcal{F}$ is a connected graph, then $1\leq CC(G) \leq n$.
\end{corollary}
 As before, we know that $K_{1}$ attains the lower bound of Corollary \ref{cor1}, while the 
complete graphs $K_{n}$ and the complete bipartite graphs $K_{r,s}$
for $2\leq r \leq s$ such that $r+s=n$, attain the upper bound.


\begin{corollary}
      If $G\not\in\big\{\mathcal{F},K_{n}\big\}$ is a connected graph of order $n$ with $k$ vertices
of degree $n-1$, then $CC(G)\ge k+2\ge 3$.
\end{corollary}

 \section{ Graphs with pendant edges}

In the this section, we  will discuss about  the connected coalition number of graphs with~$\delta(G)=1$. First we have the following results.

\begin{lemma}
\label{lemcc}
For a connected graph $G$, assume that $\psi$ is a $CC(G)$-partition. Let~$x$ be a pendant vertex and $y=\spo(x)$.  Let  $A\in \psi$ with $y\in A$. If any two sets $C,D\in \psi$ form a connected coalition, then $C=A$ or $D=A$. 
\end{lemma}
\begin{proof}
Suppose on contrary that $C\neq A$ and $D\neq A$. Since $C$ and $D$ form a connected coalition, then $C\cup D$ is a connected dominating set.  If  $C\cup D$ has no neighbor of $x$, then $x$ is not dominated by $C\cup D$. Hence, $C$ and $D$ do  not form a connected coalition., which is a contradiction. So, $C=A$ or $D=A$.  \qed
\end{proof}
\begin{lemma}
\label{lemsup}
Let $G=(V, E)$ be a  connected graph with no full vertex and  with $\delta(G)=1$ and $CC(G)\geq 3$. Let~$x$ be a pendant vertex and $y=\spo(x)$. Let $\psi$ be a $CC(G)$-partition. If $A\in \psi$ with $y\in A$, then  for any pendant vertex $w\in V$, it holds that $\spo(w)\in A$.  
\end{lemma}
\begin{proof}
Let $w$ be an arbitrary pendant vertex of $G$. Let $z=\spo(w)$. If $z\in A$, then we are done. Then, we assume that $z\not \in A$. Let $B\in\psi$ with $z\in B$.  Now, suppose that $x\not\in A$. Hence, there is a set $X\in\psi$ with $x\in X$. By Lemma \ref{lemcc}, $X$ and $A$ form a connected coalition.   Since $z\in B$, based on Lemma \ref{lemcc},  the sets $X$ and $A$ do not form a connected coalition, which is a contradiction Then, we have $x\in A$.  Since $z\in B$, based on Lemma \ref{lemcc}, for any two sets $C,D\in \psi$ forming a connected coalition, we have $C=B$ or $D=B$, also by Lemma \ref{lemcc}, since $y\in A$, $C=A$ or $D=A$. Hence, we easily conclude that $CC(G)=2$, which is a contradiction, since we assumed that $CC(G)\geq 3$. Thus, we must have $z\in A$. \qed
\end{proof}

We recall the definition of corona product of graphs. The corona product of two graphs $H_1$ and $H_2$, denoted by $H_1\circ H_2$, is defined as the graph obtained by taking one copy of $H_1$ and $|V(H_1)|$ copies of $H_2$ and joining the $i$-th vertex of $H_1$ to every vertex in the $i$-th copy of $H_2$.
 In the following, we compute the connected coalition number of connected graphs of the form
 $H\circ K_1$. To aid our discussion, we  state and prove the following theorem.

\begin{theorem}\label{4.6}
If  $G$ is a connected graph of the form $H\circ K_1$, then $CC(G)=2$.
\end{theorem}
\begin{proof}
Let $\psi$=\{$V_1$, $V_2$,\ldots, $V_k$\} be a $CC(G)$-partition of $G$. By Theorem  \ref{ge2}, we have $CC(G)\geq 2$. It suffices to prove that $CC(G)\leq 2$. Suppose on
the contrary that $CC(G)\geq 3$. Let~$x$ be a pendant vertex and $y=\spo(x)$. Let $\psi$ be a $CC(G)$-partition. Let $A\in \psi$ with $y\in A$. Then, by Lemma \ref{lemsup}, for any pendant vertex $w\in V$, it holds that $\spo(w)\in A$.  Hence, all vertices $v$ of $G$ with $deg(v)\geq 2$ lie in $A$. Then, $A$ is a connected dominating set of $G$, which is a contradiction. Hence, $CC(G)\leq 2$, and  since $CC(G)\geq 2$, we have $CC(G)=2$.
\qed
\end{proof}

We close this section with the following result.
\begin{theorem}
\label{thm:lessthann}
If $G$ is a connected graph of order $n$ with $\delta(G)=1$ and with no full vertex, then $CC(G)<n$.
\end{theorem}
\begin{proof}
 Let $\psi$ be a $CC(G)$-partition and  $v$ be a pendant vertex of $G$ and $u$ be the support vertex of $v$. Suppose on
the contrary that $CC(G)=n$. So, it must be the case that $\psi$ is a singleton partition such that every set $V_i$, for $1 \leq i\leq n$, contains a single vertex. Then, $\{v\}\in \psi$ and $\{u\}\in \psi$. If $v$ and $u$ form a connected coalition, since $v$ is a single vertex of degree one and the only neighbor of $v$ is the vertex $u$, so $u$ is adjacent to all remaining vertices, it follows that $u$ is a full vertex, a contradiction. Next assume that $v$ and $u$ do not form a connected coalition. Consequently, every singleton set $V_i$, for $1 \leq i\leq n$, must contain a full vertex, and it is a contradiction. Hence, $CC(G)<n$.\qed

\end{proof}

\section{Trees}  
 In this section, we determine the connected coalition  trees. First we have the following theorem.
 \begin{theorem}
\label{thmcct2}
For any tree $T$ of order $n$ with no full vertex, we have $CC(T)= 2$.
\end{theorem}

\begin{proof}
By Theorem  \ref{ge2}, we have $CC(T)\geq 2$. It suffices to prove that $CC(T)\leq 2$. Suppose on  the contrary that $CC(T)\geq 3$. Now we may assume that  $a$ and $b$ are two vertices of $T$  such that $a$ is a leaf and $b$ is a support vertex of $a$. Let $\psi$ be a $CC(T)$-partition, and  suppose that $V_1\in \psi$ with $b\in V_1$. Since $CC(T)\geq 3$, without loss of generality, assume that $V_2, V_3\in \psi$ are two distinct sets such that $V_2\neq V_1$ and $V_3\neq V_1$.  By Lemma \ref{lemcc}, each of $V_2$ and $V_3$ form a connected coalition with $V_1$, however, $V_2$ and $V_3$ do not form a connected coalition. Now, we consider the following cases.
\begin{itemize}

\item {\bf $T[V_1]$ is connected.} By Definition \ref{2.2}, $V_1$ is not a dominating set. Then,  there exists a vertex $u\not\in V_1$ with no neighbor in $V_1$. Hence, if any set $A\in\psi $ is in connected coalition with  $V_1$, then $A\cap N[u]\neq \emptyset.$ Assume w.l.o.g.  that $u\in V_3$. Let $u_1\in N(u)$ and assume w.l.o.g. $u_1\in V_2$. Since $T[V_1\cup V_2]$ is connected, there is a path  $P_{u_1x}$ between $u_1$ and $x$ for some vertex  $x\in V_1$. Note that all vertices on $P_{u_1x}$ are inside $V_1\cup V_2$. Also, since  $T[V_1\cup V_3]$ is connected,  there is a path  $Q_{yu}$ between $y$ and $u$ for some vertex $y\in V_1$. Note that all vertices on $Q_{yu}$ are inside $V_1\cup V_3$ (see Figure \ref{case1}).  Since $T[V_1]$ is connected, there is a path $R_{xy}$ between $x$ and $y$ inside $V_1$. Since $u_1\in N(u)$,  there is a cycle $uu_1P_{u_1x} R_{xy}Q_{yu}$ in $T$, which is a contradiction. 
\begin{figure}[ht]
\begin{center}
	\includegraphics[width=0.6\linewidth]{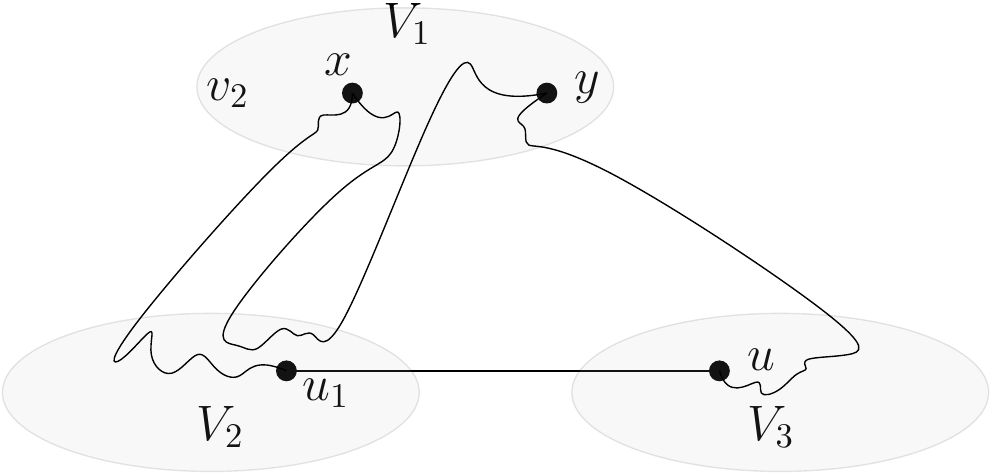}
	\caption{The case that $T[V_1]$ is connected.}
  \label{case1}
	\end{center}
\end{figure}

\item  {\bf $T[V_1]$ is not connected.} Assume that $x,y\in V_1$ such that there is no path between them in $G[V_1]$.  Since $G[V_1\cup V_2]$ is connected, there is a path $P_{x,y}$ between $x$ and $y$ that lies in $G[V_1\cup V_2]$. Also, since $G[V_1\cup V_3]$ is connected, there is a path $Q_{x,y}$ between $x$ and $y$ that lies in $G[V_1\cup V_3]$. Hence, it is clear that there are two paths between $x$ and $y$ in $T$, which is a contradiction.\qed
\end{itemize}
 \end{proof}
 As an immediate consequence of Theorem \ref{thmcct2}, we have the following result for the paths.

\begin{corollary}\label{path}

For any path $P_n$ of order $n$, where $n\ne 3$, we have $CC(P_{n})=2$.

\end{corollary}


\section{Graphs $G$ with $CC(G)=n$ and $CC(G)=n-1$}
For a given graph $G$, computing $CC(G)$ seems to be an NP-hard problem,  and therefore, computing $CC(G)$ for a class of graphs  in polynomial time seems to be interesting. In this section, we present two polynomial-time algorithms that  for a  given connected graph $G$ of order $n$ determine whether it holds   $CC(G)=n$ or $CC(G)=n-1$. For the sake of simplicity, we assume that $G$ has no full vertex.
\subsection{Graphs with $CC(G)=n$}
Let $e_{pq}$ be an edge of $G$ with two end vertices $p$ and $q$. A vertex $x\in V $ is called {\em edge-dominated} by the edge $e_{pq}$, if $x$ is adjacent to $p$ or $q$. Now, we define  the {\em edge-domination matrix} ${\cal E}_{m\times n}$ with $m$ rows and $n$ columns on the graph $G$, where $m$ is number of the edges of $G$. The definition is as follows.
$${\cal E}(e_{pq}, x)=\left\{
\begin{array}{cc}
1&\text{if the vertex~} x \text{~is edge-dominated by the edge~} e_{pq},\\
0&\text{otherwise.}
\end{array}
\right.$$
For example, the matrix  ${\cal E}$ depicted in Figure \ref{mat} is the edge-dominated matrix of the graph $C_6$ depicted in Figure \ref{e2pic}.
\begin{figure}
\begin{minipage}[b]{.55\textwidth}
\centering
\label{M}
\(
 {\cal E}=\begin{pmatrix}
1 & 1 & 1 & 0 & 0 & 1 \\
1 & 1 & 1 & 1 & 0 & 0 \\
0 & 1 & 1 & 1 & 1 & 0 \\
0 & 0 & 1 & 1 & 1 & 1 \\
1 & 0 & 0 & 1 & 1 & 1 \\
1 & 1 & 0 & 0 & 1 & 1 
\end{pmatrix}
\)
\caption{The edge-dominated matrix ${\cal E}$ for~$C_6$.}\label{mat}
\end{minipage}
\hspace{2cm}
\centering
\begin{minipage}[b]{.25\textwidth}
\includegraphics[width=\textwidth]{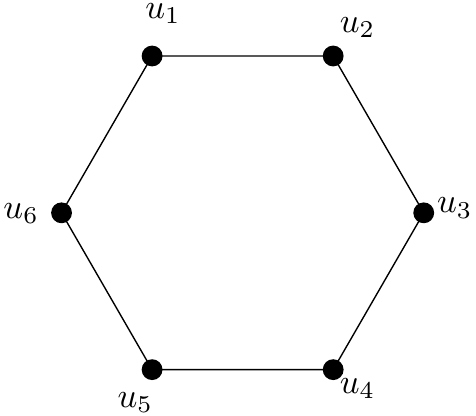}
\caption{$C_6$.}\label{e2pic}
\end{minipage}
\end{figure}

A useful matrix in Graph Theory, is the {\em incidence matrix}. For the graph $G$, the incidence  matrix $\ve_{n\times m}$ with $m$ rows and $n$ columns is defined as follows.
$${\ve}(x,e)=\left\{
\begin{array}{cc}
1&\text{if the vertex~} x \text{~is incident to the edge~} e,\\
0&\text{otherwise.}
\end{array}
\right.$$
Now, we prove the following theorem. 
\begin{theorem}
\label{thmn}
For any connected graph $G$ of order $n$ and with no full vertex, $CC(G)=n$ if and only if for any vertex $x\in V$, there is an edge $e$ with $\ve(x,e)=1$ such that
$$\sum_{v\in V}{\cal E}(e,v)=n.$$
\end{theorem}
\begin{proof}
Let $V=\{v_1,\ldots, v_n\}$ be the vertices of $G$. Suppose that $CC(G)=n$. Then, there is a $CC(G)$-partition $\psi=\left\{\{v_1\}, \ldots, \{v_n\}\right\}$ such that  any $\{v_i\}$ forms a  connected coalition with some set $\{v_j\}$ with $j\neq i$. Let $x\in V$ be an arbitrary vertex. Since $\{x\}\in\psi$, there is a vertex $u\in V$ such that $\{u\}\in \psi$ forms a connected coalition with $\{x\}$.  By the definition, $\{x,u\}$ is a connected dominating set. Then,  $e=(x,u)$ is an edge of $G$, and all vertices of $G$ is dominated by $\{x,u\}$. Therefore,  $\ve(x,e)=1$ and  ${\cal E}(e,v)=1$ for any vertex $v\in V$. Hence, $\sum_{v\in V}{\cal E}(e,v)=n$. The proof of the converse, is straightforward.\qed
\end{proof}

Now, we will describe the algorithm. The algorithm first computes the matrices $\cal E$ and $\ve$ for the graph $G$.  Then,  for all vertices  $x\in V$ the following operations is applied. For all edges $e$, the algorithm checks whether $\ve(x,e)=1$.  If  $\ve(x,e)=1$ and $\sum_{v\in V}{\cal E}(e,v)=n$, then we consider  $f=1$, and the algorithm checks another vertex of $V$.  In the algorithm, we used two variables $f$ and $flag$ to determine which vertices satisfy the conditions of Theorem \ref{thmn}. For more details, see Algorithm \ref{alg1}.
\begin{algorithm}[htb]
\scriptsize
\label{alg1}	
\SetKwInOut{Input}{input}
\Input{A connected graph $G$ with no full vertex, and with vertex set $V$ and the edge set $E$.}
 \SetKwInOut{Output}{output} 
Computes the martices $\cal E$ and $\ve$\;
$f=0$\;
$s:=0$\;
\ForEach {$x\in V$}
{
\ForEach{$e\in E$}
{
\If{ $\ve(x,e)==1$}
{
\ForEach{$v\in V$}
{
$s=s+{\cal E}(e,v)$\;
}
\If{$s==n$}
{
$f=1$\;
break\;
}
}
}
\If{$f=0$}
{
$flag:=0$\;
break;
}
\Else
{
$flag:=1$\;
$f=0$\;
}
}
\If{$flag=1$}
{
\Return {\bf yes}\;
}
\Else
{
\Return {\bf no}\;
}
\caption{{\sc CheckCCG$n$}($G, V, E$)}
\end{algorithm}

Now, we compute the time complexity of algorithm {\sc CheckCCG$n$}($G, V, E$).  It is clear that the computations of the matrices  $\cal E$ and $\ve$ take $O(mn)$ times. Then, since we have three {\bf foreach} loops, then according to the algorithm, the overall running time of  three loops is $O(n^2m)$. Hence, the overall running time of the algorithm is $O(n^2m)+O(nm)=O(n^2m)$. Since $m\in O(n^2)$, then the time complexity of the algorithm is $O(n^4)$.
Hence, we have the following theorem.
\begin{theorem}
The worst-case time complexity of algorithm {\sc CheckCCG$n$}($G, V, E$) is $O(n^4)$.
\end{theorem}
\subsection{Graphs with $CC(G)=n-1$}
Let $p=(a,b,c)$ be a triple of vertices $a$, $b$ and $c$. A vertex $x\in V $ is called {\em three-vertex-dominated} by  $p$, if $x$ is dominated by $\{a,b,c\}$. Now, we define  {three-vertex-dominated} matrix $\cal H$ as follows.
$${\cal H}\left(\{a,b,c\}, x\right)=\left\{
\begin{array}{cc}
1&\text{if the vertex~} x \text{~is  dominated by~}\{a,b,c\} ,\\
0&\text{otherwise.}
\end{array}
\right.$$
Now, we prove the following theorem.
\begin{theorem}
\label{thmfinal}
For any connected graph  $G$ of order $n$ and with no full vertex, $CC(G)=n-1$ if and only if there are two vertices $u,v\in V$ such that for any vertex $x\in V\backslash\{u,v\}$, 
\begin{enumerate}
\item there is an edge $e=(p,q)$ with $p,q\not\in\{u,v\}$  and $\ve(x,e)=1$ such that
$\sum_{v\in V}{\cal E}(e,v)=n, or$
\item $G[x,u,v]$ is connected and $\sum_{w\in V}{\cal H}(\{x,u,v\},w)=n,$
\end{enumerate}
and 
there is a vertex $y\in V\backslash\{u,v\}$ such that  $G[y,u,v]$ is connected and $$\sum_{w\in V}{\cal H}(\{y,u,v\},w)=n.$$
\end{theorem}
\begin{proof}
Let $V=\{v_1,\ldots, v_n\}$ be the vertices of $G$. Suppose that $CC(G)=n-1$. Then, there is a $CC(G)$-partition $\psi=\left\{\{w_1\}, \ldots, \{w_{n-2}\}, \{u,v\}\right\}$. Let $C\in\psi$ be an arbitrary set. Suppose that $C=\{x\}$ is singleton. If $C$ forms a connected coalition with a set $\{a\}\in \psi$, then by the definition, $\{x,a\}$ is a connected dominating set. Then,  $e=(x,a)$ is an edge of $G$, and all vertices of $G$ is dominated by $\{x,a\}$. Therefore,  $\ve(x,e)=1$ and  ${\cal E}(e,v)=1$ for any vertex $v\in V$. Hence, $\sum_{v\in V}{\cal E}(e,v)=n$. Now, if $C$ forms a connected coalition with the set $\{u,v\}\in \psi$, then by the definition, $\{x,u,v\}$ is a connected dominating set. Therefore, $G[x,u,v]$ is connected, and $\sum_{w\in V}{\cal H}(\{x,u,v\},w)=n.$ 

Now, suppose that $C=\{u,v\}$. Then, there is a vertex $\{y\}\in\psi$ that forms a connected coalition with $C$. Therefore, by the definition, $\{y,u,v\}$ is a connected dominating set. Then, $G[y,u,v]$ is connected, and $\sum_{w\in V}{\cal H}(\{y,u,v\},w)=n.$  The proof of the converse, is straightforward.\qed
\end{proof}

Now, our second algorithm depicted in Algorithm \ref{alg2}. The algorithm based on Theorem \ref{thmfinal}.
\begin{algorithm}[htb]
\scriptsize
\label{alg2}	
\SetKwInOut{Input}{input}
\Input{A connected graph $G$ with no full vertex, and with vertex set $V$ and the edge set $E$.}
 \SetKwInOut{Output}{output} 
Computes the martices $\cal H$, $\cal E$,  and $\ve$\;
$f=0$\;
$s:=0$\;
\ForEach {$u\in V$}
{
\ForEach{$v\in V$ with $u\neq v$}
{
\ForEach{$x\in V\backslash\{u,v\}$}
{
\ForEach{$e\in E$ with $\ve(x,e)=1$}
{
\If{$G[\{x,u,v\}]$ is connected and $\sum_{w\in V}{{\cal H}(\{x,u,v\},w)}=n$, or $\sum_{w\in V}{{\cal E}(e,w)}=n$  }
{
$f=1$\;
break\;
}
}
\If{$f=0$}
{
$flag:=0$\;
break\;
}
\Else
{
$flag=1$\;
$f=0$\;
}
}
\If{$flag=1$}
{
\Return yes\;
}
}
\If{$flag=1$}
{
\Return yes\;
}
}
\If{$flag=1$}
{
\Return yes\;
}
\Else
{
\Return no\;
}
\caption{{\sc CheckCCG2}($G, V, E$)}
\end{algorithm}

It is not hard to see that algorithm {\sc CheckCCG2}($G, V, E$) has  four {\bf foreach} loops and two summations. Then,  the overall running time of the algorithm is $O(n^6)$. Now, we have the following result.
\begin{theorem}
The worst-case time complexity of algorithm {\sc CheckCCG2}($G, V, E$) is $O(n^6)$.
\end{theorem}

\section{Conclusion and future works}
In this paper, we have introduced the connected coalition concept in graphs and we have studied some properties for the connected  coalition number. We characterized all graphs whose have a connected coalition partition. We have shown that  for any graph $G$ with $\delta(G)=1$ and with no full vertex, $CC(G)\leq n-1$. Also we proved that for any tree $T$, $CC(T)=2$. Finally, we have presented two polynomial-time algorithms that for a given connected graph $G$ of order $n$ determine whether $CC(G)=n$ or $CC(G)=n-1$. 

 There are many open problems  in the study of the connected coalition number  of a graph that we  state and close the paper with some of them. 
\begin{enumerate}
	\item What is the connected  coalition number of  graph operations, such as corona, Cartesian, join, lexicographic, and so on?
	
	\item What is the connected  coalition number of  natural and fractional powers of a graph (see e.g. \cite{BIMS})?
	
	\item What is the effects on $CC(G)$ when $G$ is modified by operations on vertex and edge of $G$?
	\item Similar to the coalition graph of $G$, it is natural to define and study the connected  coalition graph of $G$ for connected  coalition partition $\pi$, which can be denoted by $CCG(G,\pi)$, and is defined as follows. Corresponding to any connected coalition partition $\pi=\{V_1,V_2,\ldots, V_k\}$ in  a graph $G$, a  {\em connected coalition graph} $CCG(G, \pi)$ is  associated in which there is a one-to-one correspondence between the   vertices of $CCG(G,\pi)$  and the sets $V_1, V_2,...,V_k$ of $\pi$, 
	and two vertices of  $CCG(G,\pi)$  are adjacent if and only if their corresponding
	sets in $\pi$ form a connected  coalition. 
	
	
\end{enumerate} 
\medskip

\noindent{\bf Acknowledgement.} 
The work of Hamidreza Golmohammadi is supported by the Mathematical Center in Akademgorodok, under agreement No. 075-15-2022-282 with the Ministry of Science and High Education of the Russian Federation.

\end{document}